\documentclass[10pt,reqno]{amsart}
\usepackage{amsmath,amsopn,amssymb,amsthm}
\usepackage[cp1251]{inputenc}
\usepackage[english]{babel}
\usepackage[pagebackref,breaklinks=true,colorlinks=true,linkcolor=blue,citecolor=blue,urlcolor=blue]{hyperref}
\usepackage{graphicx}%
\usepackage{color}

\newtheorem{theorem}{Theorem}[section]

\newtheorem{lemma}[theorem]{Lemma}

\newtheorem{question}[theorem]{Question}

\newtheorem{remark}[theorem]{Remark}

\renewenvironment{proof}[1][Proof]{\noindent\textbf{#1.} }{\ \rule{0.5em}{0.5em}}

\begin{document}
\title[Parallel translations for a left invariant spray]{Parallel translations for a left invariant spray}
\author{Ming Xu}

\address{Ming Xu \newline
School of Mathematical Sciences,
Capital Normal University,
Beijing 100048,
P. R. China}
\email{mgmgmgxu@163.com}

\date{}

\begin{abstract}
In this paper, we study the left invariant spray geometry on a connected Lie group. Using the technique of invariant frames, we
find the ordinary differential equations on the Lie algebra describing for a left invariant spray structure the linearly parallel translations along a geodesic and the nonlinearly parallel translations along a smooth curve. In these equations, the connection operator plays an important role. Using linearly parallel translations, we provide alternative interpretations or proofs for some homogeneous curvature formulae. Concerning the nonlinearly ones,
we propose two questions in left invariant spray geometry. One question generalizes Landsberg Problem in Finsler geometry, and the other concerns the restricted holonomy group.

\textbf{Mathematics Subject Classification (2010)}: 53B40, 53C30, 53C60

\textbf{Key words}: connection operator, geodesic, holonomy group, invariant frame, invariant spray structure, parallel translation, Riemann curvature, spray vector field

\end{abstract}
\maketitle

\section{Introduction}
{\it Spray geometry} concerns a {\it spray structure} $\mathbf{G}$ on a smooth manifold $M$, which is a smooth vector field on $TG\backslash\{0\}$ with the standard local coordinate presentation
$\mathbf{G}=y^i\partial_{x^i}-2\mathbf{G}^i\partial_{y^i}$,
where $\mathbf{G}^i=\mathbf{G}^i(x,y)$ is positive 2-homogeneous for the $y$-entry \cite{Sh2001-1}. For example, when $\mathbf{G}$ is induced by a Finsler metric $F$, then
$\mathbf{G}^i=\tfrac14g^{il}([F^2]_{x^ky^l}y^k-[F^2]_{x^l})$ \cite{Sh2001-2}.
Many notions in Finsler geometry, like (nonconstant) geodesic, Riemann curvature,
linear and nonlinear parallel translations, are only relevant to $\mathbf{G}$, i.e., they are originated from spray geometry \cite{Be1926}. See \cite{HM2021}\cite{LMY2019}\cite{LS2018}\cite{Ya2011} for some recent progress in this field.

In this paper, we discuss a special class of homogeneous spray structure. A spray structure $\mathbf{G}$ on a connected Lie group $G$ is called left invariant, if it is preserved by all left translations \cite{HM2021}\cite{Xu2021}. Using a left invariant frame $\{\widetilde{U}_i,\partial_{u^i},\forall i\}$ on $TG$ (see Section \ref{subsection-2-2} below or Section 3.1 in \cite{Xu2021}), a left invariant spray structure
$\mathbf{G}$ can be presented as $\mathbf{G}=\mathbf{G}_0-\mathbf{H}$, where
$\mathbf{G}_0=u^i\widetilde{U}_i$ is the canonical bi-invariant spray structure, and $\mathbf{H}=\mathbf{H}^i\partial_{u^i}$ is a left invariant vector field on $TG\backslash0$ which is tangent to each $T_gG$. The restriction $\eta=\mathbf{H}|_{T_{e}G\backslash\{0\}}$ is called the
{\it spray vector field} associated to $\mathbf{G}$ \cite{Xu2021}. This notion was first proposed by L. Huang in homogeneous Finsler geometry \cite{Hu2015-1}. We usually present $\eta$ as smooth function from $\mathfrak{g}\backslash\{0\}$ to $\mathfrak{g}$ (see Section 2.3 below).

The philosophy of homogeneous geometry implies that, to explore a left invariant $\mathbf{G}$, we only need to observe the interaction between the dynamical system of $\eta$ and the Lie algebra structure of $\mathfrak{g}$. Following this thought, we find homogeneous curvature formulae for $\mathbf{G}$ (see (\ref{006}) below and Corollary 4.1 in \cite{Xu2021}) which generalizes those of L. Huang \cite{Hu2015-1}\cite{Hu2015-2} in homogeneous Finsler geometry,
and prove a correspondence between geodesics of $\mathbf{G}$ and  integral curves of $-\eta$ (see Theorem \ref{thm-2} below or Theorem D in \cite{Xu2021}).

As a continuation of this exploration, we switch in this paper to linearly and nonlinearly
parallel translations for $\mathbf{G}$. It turns out that the {\it connection operator} $N:(\mathfrak{g}\backslash\{0\})\times\mathfrak{g}\rightarrow
\mathfrak{g}$ plays an important role. This notion
was first defined for a homogeneous Finsler space by L. Huang, using fundamental tensor and Cartan tensor \cite{Hu2015-1}\cite{Hu2015-2}. He later pointed out another description (see (4) in \cite{Hu2017}), which implies that it can be generalized to homogeneous spray geometry. In particular, for a left invariant spray structure $\mathbf{G}$ with the spray vector field $\eta$, the connection operator is
\begin{equation}\label{015}
N(y,w)=\tfrac12D\eta(y,w)-\tfrac12[y,w]_\mathfrak{g},
\end{equation}
in which $D\eta(y,w)$ is the derivative of $\eta:\mathfrak{g}\backslash\{0\}\rightarrow\mathfrak{g}$
at $y$ in the direction of $w$. Here $[\cdot,\cdot]_\mathfrak{g}$ is the Lie bracket of $\mathfrak{g}=\mathrm{Lie}(G)$. The more usual notation $[\cdot,\cdot]$ is reserved for the canonical bracket between two smooth vector fields (see Theorem \ref{thm-1} and Remark \ref{remark} below).

Firstly, we consider linearly parallel translations along a smooth curve on $(G,\mathbf{G})$ and prove
\begin{theorem}\label{main-thm-1}
Let $G$ be a connected Lie group endowed with a left invariant spray structure $\mathbf{G}$  with the spray vector field $\eta$,
$c(t)$ a smooth curve on $(G,\mathbf{G})$ with nowhere-vanishing $\dot{c}(t)$, and $W(t)$ a vector field along $c(t)$. Denote $y(t)=(L_{c(t)^{-1}})_*(\dot{c}(t))\in\mathfrak{g}\backslash\{0\}$
and $w(t)=(L_{c(t)^{-1}})_*(W(t))\in\mathfrak{g}$. Then $W(t)$ is linearly parallel along $c(t)$ if and only if $w(t)$ is a solution of
\begin{equation}\label{011}
\tfrac{{\rm d}}{{\rm d}t}w(t)+N(y(t),w(t))+[y(t),w(t)]_\mathfrak{g}=0.
\end{equation}
\end{theorem}

As an application of Theorem \ref{main-thm-1}, the connection operator and Riemann curvature operators for a left invariant $\mathbf{G}$
can be alternatively described by
\begin{theorem} \label{thm-1}
Let $c(t)$ be a geodesic on the connected Lie group $G$, for the left invariant spray structure $\mathbf{G}$, and $y(t)=(L_{c(t)^{-1}})_*(\dot{c}(t))$ the corresponding integral curve for $-\eta$, where $\eta$ is the spray vector field associated with $\mathbf{G}$. Denote $w(t)$ the vector field along $y(t)$ which corresponds
to a linearly parallel vector field $W(t)$ along $c(t)$, i.e., $w(t)=(L_{c(t)^{-1}})_*(W(t))$. Then $N(t)=N(y(t),w(t))$ and $R(t)=R_{y(t)}(w(t))$, where $N(\cdot,\cdot)$ is the connection operator and $R_\cdot\cdot$ is the Riemann curvature, are vector fields along $y(t)$ determined by
\begin{equation} \label{100}
N(t)=-[\eta,w(t)]\quad\mbox{and}\quad R(t)=[\eta,N(t)]=-[\eta,[\eta,w(t)]].
\end{equation}
\end{theorem}

\begin{remark}\label{remark}
The vector fields $w(t)$ and $N(t)$ are smooth vector field along an integral of $-\eta$. So $[\eta,w(t)]=-[-\eta,w(t)]$ and $[\eta,N(t)]=-[-\eta,N(t)]$ in (\ref{100}) are in deed Lie derivatives. Notice that they can be calculated as the canonical bracket between smooth vector field on a manifold, as they are presented in Theorem \ref{thm-1}.
To be precise, let $X$ be a smooth vector field on $M$ and $Y(t)$ a smooth vector field along an integral curve $c(t)$ for $X$, then $[X,Y(t)]$ is a well defined smooth vector field along $c(t)$. When $c(t)$ is not constant, we may locally extend $Y(t)$ to a smooth vector field $Z$ on $M$, then $[X,Y(t)]=[X,Z]|_{c(t)}$ is independent of the extension. Using local coordinate, the bracket between $X=X^i\partial_{x^i}$ and
$Y(t)=Y^i(t)\partial_{x^i}|_{c(t)}$ can be presented as
\begin{equation}\label{007}
[X,Y(t)]=(\tfrac{{\rm d} Y^i(t)}{{\rm d}t}\partial_{x^i}-Y^i(t)\tfrac{\partial}{\partial x^i}X^j\partial_{x^j})|_{c(t)}.
\end{equation}
Notice that when $c(t)$ is constant, (\ref{007}) can still be used to
calculate $[X,Y(t)]$, which is independent of the choice of
local coordinate.
\end{remark}

Theorem \ref{main-thm-1} (together with Theorem D in \cite{Xu2021}) can provide shortcuts to other curvature formulae of $\mathbf{G}$ as well. See Section \ref{subsection-2-3} for some examples.

Nextly, we consider nonlinearly parallel translations along a smooth curve on $(G,\mathbf{G})$ and prove
\begin{theorem}\label{main-thm-2}
Let $G$ be a connected Lie group endowed with a left invariant spray structure $\mathbf{G}$, $c(t)$ a smooth curve on $G$ and $Y(t)$ a nowhere vanishing vector field along $c(t)$. Denote $w(t)=(L_{c(t)^{-1}})_*(\dot{c}(t))$ and
$y(t)=(L_{c(t)^{-1}})_*(Y(t))$. Then $Y(t)$ is nonlinearly parallel along $c(t)$ iff $y(t)$ is a solution of
\begin{equation}\label{012}
\tfrac{{\rm d}}{{\rm d}t}y(t)+N(y(t),w(t))=0.
\end{equation}
\end{theorem}
When $w(t)\equiv w$ is constant, i.e., $c(t)=\exp tw$ is a one-parameter subgroup of $G$, the equation (\ref{012}) generates a one-parameter subgroup of diffeomorphisms on $\mathfrak{g}\backslash\{0\}$. So Theorem \ref{main-thm-2} has the following immediate consequence.

\begin{theorem}\label{cor}
Let $G$ be a connected Lie group endowed with a left invariant spray structure $\mathbf{G}$, and $c(t)=\exp tw$ for any $w\in\mathfrak{g}$ a one-parameter subgroup of $G$.
Denote $\mathbf{P}^{\mathrm{nl}}_{c(0),c(t);c}$ the nonlinear parallel translation along $c(\cdot)$ from $c(0)$ to $c(t)$. Then
$\rho_t=(L_{c(t)^{-1}})_*\circ\mathbf{P}^{\mathrm{nl}}_{c(0),c(t);c}$ is the one-parameter subgroup of diffeomorphisms generated by the smooth vector field $-N(\cdot,w)$ on $\mathfrak{g}\backslash\{0\}$.
\end{theorem}

Denote $\mathfrak{H}$ the Lie algebra
generated by the space
$\mathcal{N}=\{N(\cdot,v),\forall v\in\mathfrak{g}\}$,
using the canonical bracket between two smooth vector fields on $\mathfrak{g}\backslash\{0\}$,
for a left invariant spray structure $\mathbf{G}$ on the Lie group $G$. The importance of $\mathfrak{H}$ in homogeneous spray and Finsler geometries is implied by
Theorem \ref{cor}, i.e., $\mathfrak{H}$ contains all the information for the nonlinear parallel translations on $(G,\mathbf{G})$.
We propose two questions (see Question \ref{que-1} and Question \ref{que-2} in Section \ref{subsection-4-2}).
The first question concerns if we always have $\dim\mathfrak{H}=+\infty$ when
$\mathbf{G}$ is not affine. This question can be viewed as a generalization of Landsberg Problem \cite{Ma1996} in Finsler geometry. The second question concerns the relation between $\mathfrak{H}$ and the restricted holonomy group $\mathrm{Hol}_0(G,\mathbf{G})$. Notice that when $\mathbf{G}$
is induced by a left invariant Riemannian metric, we can often get
$\mathfrak{H}=\mathrm{Lie}(\mathrm{Hol}_0(G,\mathbf{G}))$
\cite{Ko1955}\cite{Ko1957}. However, in Finsler or spray geometry, both $\mathfrak{H}$ and $\mathrm{Hol}_0(G,\mathbf{G})$ might have
infinite dimensions \cite{HMM2020}, and then the second question becomes much harder.

This paper is organized as following. In Section 2, we summarize some necessary notions and techniques. In Section 3, we prove Theorem \ref{main-thm-1} and Theorem \ref{thm-1}. In Section 4, we prove Theorem \ref{main-thm-2} and propose two questions.
\section{Preliminaries}
\subsection{Spray structure and parallel translation}
In this subsection, we summarize some fundamental knowledge in  \cite{Sh2001-1} on spray geometry.

A {\it spray structure} on a smooth manifold $M$ is
a smooth tangent vector field $\mathbf{G}$ on the slit tangent bundle $TM\backslash0$, which can be locally presented as
\begin{equation}\label{000}
\mathbf{G}=y^i\partial_{x^i}-2\mathbf{G}^i\partial_{y^i}
\end{equation}
for any {\it standard local coordinate} $(x^i,y^i)$, i.e., $x=(x^i)\in M$ and
$y=y^i\partial_{x^i}\in T_xM$, such that each $\mathbf{G}^i=\mathbf{G}^i(x,y)$ is positive 2-homogeneous for the $y$-entry.
A smooth curve $c(t)$ in $M$ with nonvanishing  $\dot{c}(t)$ everywhere is called a {\it geodesic} for
$\mathbf{G}$, if its lifting $(c(t),\dot{c}(t))$ in $TM\backslash0$ is an integral curve of $\mathbf{G}$.

Following (\ref{000}),
we denote
$$N^i_j=\tfrac{\partial\mathbf{G}^i}{\partial{y^j}}
\quad\mbox{and}\quad
\delta_{x^i}=\partial_{x^i}-N^i_j\partial_{y^j}.$$

Let $W(t)=z^i(t)\partial_{x^i}|_{c(t)}$ be a smooth vector field along the smooth curve $c(t)$ and we assume $\dot{c}(t)$ is nonvanishing everywhere. The {\it linearly covariant derivative} $D_{\dot{c}(t)}W(t)$  is the following smooth vector field along $c(t)$,
\begin{equation}\label{004}
D_{\dot{c}(t)}W(t)=(\tfrac{{\rm d} z^i(t)}{{\rm d}t}+z^j(t)N^i_j(c(t),\dot{c}(t)))\partial_{x^i}|_{c(t)}.
\end{equation}
We say $W(t)$ is {\it linearly parallel} along $c(t)$ if $D_{\dot{c}(t)}W(t)\equiv0$. For any initial value $w\in T_{c(t_0)}M$, there exists a unique linearly parallel $W(t)$
globally along $c(t)$. This fact is implies by the existence and uniqueness theorem for ordinary differential equation with initial value, together with the smoothness and positive 1-homogeneity
of $N^i_j$'s.

Suppose $c(t)$ with $t\in[a,b]$, $c(a)=p$ and $c(b)=q$ is a smooth curve on $M$ with nonvanishing $\dot{c}(t)$ everywhere. Then
the {\it linear parallel translation} $\mathbf{P}^{\mathrm{l}}_{p,q;c}:T_pM\rightarrow T_qM$ is defined by $\mathbf{P}^{\mathrm{l}}_{p,q;c}(w)=W(b)$, where $W(t)$ is the linearly parallel vector field along $c(t)$ satisfying $W(a)=w$.

Using the {\it nonlinearly covariant derivative}
$$\widetilde{D}_{\dot{c}(t)}W(t)=(\tfrac{{\rm d}z^i(t)}{{\rm d}t}+
\dot{c}^j(t)
N^i_j(c(t),W(t)))\partial_{x^i}|_{c(t)}$$
for any nowhere vanishing $W(t)=z^i(t)\partial_{x^i}|_{c(t)}$ along $c(t)$, the {\it nonlinearly parallel translation} $\mathbf{P}^{\mathrm{nl}}_{p,q;c}:T_pM\backslash\{0\}\rightarrow T_qM\backslash\{0\}$ from $p=c(a)$ to $q=c(b)$ along a smooth curve $c(t)$ can be similarly defined.
It can be alternatively described by the integral curves for the horizonal lifting of
$\dot{c}(t)$, i.e.,
$$\widetilde{\dot{c}(t)}^\mathcal{H}=
\dot{c}^i(t)\delta_{x^i}|_{T_{c(t)}M\backslash\{0\}}=
\dot{c}^i(t)(\partial_{x^i}-N^i_j\partial_{y^j})|_{T_{c(t)}M
\backslash\{0\}}.$$
To be precise, $\widetilde{\dot{c}(t)}^\mathcal{H}$
is viewed as a smooth tangent vector field on the submanifold $S=\cup_{t}(T_{c(t)}M\backslash\{0\})$.  For any nowhere vanishing smooth vector field $Y(t)$ along $c(t)$, it is {\it nonlinearly parallel}, i.e.,
$Y(t)=\mathbf{P}^{\mathrm{nl}}_{c(t_0),c(t);c}(Y(t_0))$ for all values of $t$, iff $(c(t),Y(t))$ is an integral curve of $\widetilde{\dot{c}(t)}^\mathcal{H}$ in $S$.

\subsection{Invariant frames on a Lie group}
\label{subsection-2-2}
In this section, we summarize some notations in \cite{Xu2021} for invariant frames.

Let $G$ be a connected Lie group, $L_g(g')=gg'$ and $R_g(g')=g'g$ the left and right translations respectively. Denote $\mathfrak{g}=\mathrm{Lie}(G)=T_eG$ and $[\cdot,\cdot]_{\mathfrak{g}}$ its Lie bracket.
We fix a basis $\{e_1,\cdots,e_n\}$ of $\mathfrak{g}$ with $[e_i,e_j]_{\mathfrak{g}}=c^k_{ij}e_k$.

For each $1\leq i\leq n$, we have a left invariant tangent vector field
$U_i(g)=(L_g)_*(e_i)$, so $\{U_i,\forall i\}$ is a left invariant frame on $G$ satisfying $[U_i,U_j]=c^k_{ij}U_k$. The complete lifting of $U_i$ is denoted
as $\widetilde{U}_i$, which is a smooth tangent vector field on $TG$
(see Section 2.1 in \cite{Xu2021}).
Any tangent vector $y\in T_gG$ can be
uniquely presented as $y=u^i U_i(g)$, which determines the smooth functions $u^i$'s on $T_gG$. Denote $\partial_{u^i}$'s the smooth vector fields on $TG$ which are tangent to and correspond to the $u^i$-coordinates in each $T_gG$. It is easy to observe the left invariancy of $U_i$'s and $u^i$'s, so
we call $\{\widetilde{U}_i,\partial_{u^i},\forall i\}$ a {\it left invariant frame on} $TG$.

The transformation between $\{\widetilde{U}_i,\partial_{u^i},\forall i\}$ and $\{\partial_{x^i},\partial_{y^i},\forall i\}$ for a standard local coordinate $(x^i,y^i)$ is the following (see (2.2) and Lemma 2.1 in \cite{Xu2021}),
\begin{eqnarray}
& &U_i=A^j_i\partial_{x^i},\quad u^i=y^jB^i_j\quad\mbox{and}\quad
\partial_{u^i}=A^j_i\partial_{y^i},\label{001}\\
& &\widetilde{U}_i=A^j_i\partial_{x^i}+y^j\tfrac{\partial}{\partial x^j}A^k_i\partial_{y^i},\label{002}
\end{eqnarray}
where $(A^j_i)=(A^j_i(x))$ and $(B^j_i)=(B^j_i(x))=(A^j_i(x))^{-1}$ (i.e., $A^j_iB^k_j=B^j_iA^k_j=\delta^k_i$) are matrix valued
functions which only depend on the $x$-entry. Notice that
$\partial_{x^i}$'s in (\ref{001}) and (\ref{002}) are local tangent vector fields on $G$, and their complete liftings to $TG$ respectively.
Comparing the coefficients of $\partial_{x^k}$ in both sides of
$$c^p_{qi}A^k_p\partial_{x^k}=c^p_{qi}U_p=[U_q,U_i]=
(A^j_q\tfrac{\partial}{\partial x^j}A^k_i-A^j_i\tfrac{\partial}{\partial x^j}A^k_q)\partial_{x^k},$$
we get (see (3.14) in \cite{Xu2021})
\begin{lemma}\label{lemma-0}
$A^j_q\tfrac{\partial}{\partial x^j}A^k_i-A^j_i\tfrac{\partial}{\partial x^j}A^k_q=c^p_{qi}A^k_p$.
\end{lemma}
Using the right invariant tangent vector fields $V_i(g)=(R_g)_*(e_i)$ on $G$ and the presentation
$y=v^i V_i(g)\in T_gG$, we can similarly get a {\it right invariant
frame} $\{\widetilde{V}_i,\partial_{v^i},\forall i\}$ {\it on} $TG$.

To describe the transformation between $\{\widetilde{U}_i,\partial_{u^i},\forall i\}$ and $\{\widetilde{V}_i,\partial_{v^i},\forall i\}$,
we denote $\phi_i^j$ and $\psi_i^j$ the functions on $G$ such that $\mathrm{Ad}(g)e_i=\phi_i^je_j$ and $\mathrm{Ad}(g^{-1})e_i=\psi_i^je_j$ (so we have
$\psi_i^j\phi_j^k=\phi_i^j\psi_j^k=\delta_i^k$). Then at each $g\in G$,
$$U_i(g)=(L_g)_*(e_i)=(R_g)_*(R_{g^{-1}})_*(L_g)_*(e_i)
=(R_g)_*(\mathrm{Ad}(g)e_i)
=(R_g)_*(\phi^j_i(g) e_j)=\phi^j_i(g)V_j(g),$$
so we have
\begin{equation}\label{003}
U_i=\phi_i^jV_j,\quad u^i=\psi_j^iv^j,\quad
\partial_{u^i}=\phi_i^j\partial_{v^j}.
\end{equation}
In \cite{Xu2021}, we have proved
\begin{lemma} \label{lemma-2}
(1) $\phi^j_l V_j\phi_i^k=c^j_{li}\phi^k_j$,
\quad (2) $\widetilde{U}_i=\phi^j_i\widetilde{V}_j+c^q_{pi}u^p\partial_{u^q}$.
\end{lemma}

We briefly recall its proof here. To prove (1), we observe
\begin{eqnarray*}
V_j\phi_i^ke_k&=&V_j(\phi_i^ke_k)=\tfrac{{\rm d}}{{\rm d}t}
(\mathrm{Ad}(\exp te_j\cdot g)e_i)\\
&=&\tfrac{{\rm d}}{{\rm d}t}\mathrm{Ad}(\exp te_j)(\mathrm{Ad}(g)e_i)
=[e_j,\mathrm{Ad}(g)e_i]_\mathfrak{g},
\end{eqnarray*}
and
\begin{equation}\label{013}
\phi_l^j V_j\phi_i^k e_k=
[\mathrm{Ad}(g)e_l,\mathrm{Ad}(g)e_i]_\mathfrak{g}
=\mathrm{Ad}(g)[e_l,e_i]_\mathfrak{g}=
c^j_{li}\mathrm{Ad}(g)e_j=c^j_{li}
\phi^k_j e_k.
\end{equation}
Then (1) follows after a comparison for the coefficients of $e_k$ in (\ref{013}). To prove (2), we apply Lemma 2.2 in \cite{Xu2021} to the first equality in (\ref{003}) and get
$$\widetilde{U}_i=\phi_i^j\widetilde{V}_j+v^jV_j\phi_i^k\partial_{v^k}
=\phi_i^j\widetilde{V}_j+u^p\phi_p^jV_j\phi_i^k\partial_{v^k}.$$
Then (2) follows after (1) and the third equality in (\ref{003}) immediately.
\subsection{Left invariant spray structure on a Lie group}

In this subsection, we introduce the notions of invariant spray structure in \cite{Xu2021}.

Let $G$ be a connected Lie group and $\mathbf{G}$ a spray structure
on $G$. We call $\mathbf{G}$ {\it left invariant} (or {\it right invariant}) if all left (or right respectively) translations
preserve $\mathbf{G}$. We call $\mathbf{G}$ {\it bi-invariant} if it is both left and right invariant.
Using the invariant frames $\{\widetilde{U}_i,\partial_{u^i},\forall i\}$ and $\{\widetilde{V}_i,\partial_{v^i},\forall i\}$ on $TG$, the left and right invariancies of $\mathbf{G}$ can be equivalently described as
$$[\widetilde{V}_i,\mathbf{G}]=0,\forall i\quad\mbox{and}
\quad[\widetilde{U}_i,\mathbf{G}]=0,\forall i\quad\mbox{respectively}.$$

The canonical bi-invariant spray structure $\mathbf{G}_0=u^i\widetilde{U}_i=v^i\widetilde{V}_i$ (see Theorem A in \cite{Xu2021})
serves as the origin
in the space of left invariant spray structures on $G$. Any left invariant spray
structure $\mathbf{G}$ on $G$ can be presented as
$\mathbf{G}=\mathbf{G}_0-\mathbf{H}=\mathbf{G}_0-
\mathbf{H}^i\partial_{u^i}$. Here $\mathbf{H}=\mathbf{H}^i\partial_{u^i}$ is a left invariant smooth vector field on $TG\backslash0$, and  $\mathbf{H}^i$'s are left invariant smooth functions on $TG\backslash0$ which are positive 2-homogeneous in each $T_gG$.
We denote $\eta=\mathbf{H}|_{T_eG\backslash\{0\}}$ and
call it the {\it spray vector field} associated with $\mathbf{G}$. Usually $\eta$ is presented as a smooth map from $\mathfrak{g}\backslash\{0\}$ to $\mathfrak{g}$, i.e., $\eta(y)=\mathbf{H}^i(e,y)e_i$. The {\it connection operator}  $N(\cdot,\cdot):(\mathfrak{g}\backslash\{0\})\times\mathfrak{g}
\rightarrow\mathfrak{g}$ is then defined by $N(y,w)=\tfrac12D\eta(y,w)-\tfrac12[y,w]_\mathfrak{g}$,
in which $D\eta(y,w)$ is the derivative of the spray vector field $\eta:\mathfrak{g}\backslash\{0\}\rightarrow\mathfrak{g}$
at $y$ in the direction of $w$.

If we present the left invariant spray structure $\mathbf{G}=\mathbf{G}_0-\mathbf{H}^i\partial_{u^i}$
as $\mathbf{G}=y^i\partial_{x^i}-2\mathbf{G}^i\partial_{y^i}$ for any standard local coordinate $(x^i,y^i)$. Using (\ref{001}) and (\ref{002}), it is easy to check (see (3.10) in \cite{Xu2021})
\begin{lemma} \label{lemma-1}
 $N^k_j=\tfrac{\partial\mathbf{G}^k}{\partial y^j}
=\tfrac12A^k_i\tfrac{\partial}{\partial y^j}\mathbf{H}^i
-\tfrac12u^i\tfrac{\partial}{\partial x^j}A^k_i-
\tfrac12y^l B^i_j\tfrac{\partial}{\partial x^l}A^k_i$.
\end{lemma}
\section{Linear parallel translation along a geodesic}
\subsection{Proof of Theorem \ref{main-thm-1}}
Suppose the connected Lie group $G$ is  endowed with a left invariant spray structure $\mathbf{G}=\mathbf{G}_0-\mathbf{H}=
u^i\widetilde{U}_i-\mathbf{H}^i\partial_{u^i}$ with the spray vector field $\eta:\mathfrak{g}\backslash\{0\}\rightarrow\mathfrak{g}$. Now we consider
the linear parallel translation along a smooth curve $c(t)$ on $(G,\mathbf{G})$ with nowhere-vanishing $\dot{c}(t)$.
We denote $y(t)=(L_{c(t)^{-1}})_*(\dot{c}(t))$, which is a smooth curve on $\mathfrak{g}\backslash\{0\}$.
%
Using left translations, smooth vector fields along $c(t)$ and those along $y(t)$
can be one-to-one corresponded. Consider any smooth vector field $W(t)=w^i(t)U_i(c(t))$ along $c(t)$, then the corresponding
$w(t)=(L_{c(t)^{-1}})_*(W(t))$ along $y(t)$ can be presented as
$w(t)=w^i(t)e_i$.

\begin{lemma} \label{lemma-5}
$D_{\dot{c}(t)}W(t)=(\tfrac{{\rm d}w^l(t)}{{\rm d}t}+
\tfrac12w^j(t)\tfrac{\partial}{\partial u^j}\mathbf{H}^l+
\tfrac12w^j(t)u^k(t) c_{kj}^l )U_l(c(t))$.
\end{lemma}
\begin{proof}
In any standard local coordinate $(x^i,y^i)$, we have the presentations
$c(t)=(c^i(t))$, $y^i(t)=y^i(c(t),\dot{c}(t))=\dot{c}^i(t)$ and $W(t)=w^i(t)U_i(c(t))=z^i(t)\partial_{x^i}|_{c(t)}$. Using the notations in (\ref{001}), i.e., $U_i=A^j_i\partial_{x^i}$, $u^i=y^jB^i_j$ and $(B^i_j(x))=(A^i_j(x))^{-1}$, we also have
$u^i(t)=y^j(t)B^i_j(c(t))$ and
$z^i(t)=w^j(t)A^i_j(c(t))$. So at each point $x=c(t)$, the covariant derivative $D_{\dot{c}(t)}W(t)$ (see (\ref{004})) can be calculated as following,
\begin{eqnarray}
D_{\dot{c}(t)}W(t)
&=&(\tfrac{{\rm d}z^i(t)}{{\rm d}t}+z^j(t)N^i_j(c(t),\dot{c}(t)))
\partial_{x^i}\nonumber\\
&=&\tfrac{{\rm d}}{{\rm d}t}(w^j(t)A_j^i)
B_i^k U_k+z^j(t)(\tfrac12A^i_l
\tfrac{\partial}{\partial y^j}\mathbf{H}^l-
\tfrac12 u^l(t)
\tfrac{\partial}{\partial x^j}A^i_l-\tfrac12 y^p(t)B^l_j \tfrac{\partial}{\partial x^p}A^i_l)\partial_{x^i}\nonumber\\
&=&(\tfrac{{\rm d}w^k(t)}{{\rm d}t}U_k+
z^p(t)y^l(t) B_p^j \tfrac{\partial}{\partial x^l}A^i_j \partial_{x^i})\nonumber\\
& &
+\tfrac12w^j(t)\tfrac{\partial}{\partial u^j}\mathbf{H}^l U_l
-(\tfrac12 z^j(t)y^p(t) B^l_p\tfrac{\partial}{\partial x^j}A^i_l+\tfrac12z^j(t) y^p(t) B^l_j
\tfrac{\partial}{\partial x^p}A^i_l
)\partial_{x^i}\nonumber\\
&=&(\tfrac{{\rm d}w^k(t)}{{\rm d}t}+\tfrac12w^j(t)\tfrac{\partial}{\partial u^j}\mathbf{H}^k)U_k+\tfrac12(
z^p(t)y^l(t)B^j_p\tfrac{\partial}{\partial x^l}A^i_j-
z^j(t)y^p(t)B^l_p\tfrac{\partial}{\partial x^l}A^i_k)
\partial_{x^i}\nonumber\\
&=&(\tfrac{{\rm d}w^k(t)}{{\rm d}t}+\tfrac12w^j(t)\tfrac{\partial}{\partial u^j}\mathbf{H}^k)U_k+\tfrac12(w^j(t)u^k(t) A^l_k \tfrac{\partial}{\partial x^l}A^i_j-w^j(t)u^k(t) A_j^l \tfrac{\partial}{\partial x^l}A^i_k)\partial_{x^i}\nonumber\\
&=&(\tfrac{{\rm d}w^l(t)}{{\rm d}t}+\tfrac12w^j(t)\tfrac{\partial}{\partial u^j}\mathbf{H}^l+\tfrac12w^j(t)u^k(t) c_{kj}^l )U_l(c(t)),\label{005}
\end{eqnarray}
in which the second line uses Lemma \ref{lemma-1} and the last line uses Lemma \ref{lemma-0}.
\end{proof}\bigskip

Now we are ready to prove Theorem \ref{main-thm-1}, which  interprets Lemma \ref{lemma-5} by left translations.

\bigskip\begin{proof}[Proof of Theorem \ref{main-thm-1}]
Lemma \ref{lemma-5} indicates
\begin{eqnarray*}
(L_{c(t)^{-1}})_{*}(D_{\dot{c}(t)}W(t))
&=&(\tfrac{{\rm d}w^l(t)}{{\rm d}t}+\tfrac12w^j(t)\tfrac{\partial}{\partial u^j}\mathbf{H}^l+\tfrac12w^j(t)u^k(t) c_{kj}^l )e_l\\
&=&(\tfrac{{\rm d}w^l(t)}{{\rm d}t}+(\tfrac12w^j(t)\tfrac{\partial}{\partial u^j}\mathbf{H}^l-\tfrac12w^j(t)u^k(t)) c_{kj}^l+w^j(t) u^k(t) c_{kj}^l )e_l\\
&=&\tfrac{{\rm d}}{{\rm d}t}{w}(t)+N(y(t),w(t))+[y(t),w(t)]_\mathfrak{g},
\end{eqnarray*}
where $y(t)=u^i(t)e_i$ and $w(t)=w^i(t)e_i$.
So $D_{\dot{c}(t)}W(t)\equiv0$ if and only if $(L_{c(t)^{-1}})_{*}(D_{\dot{c}(t)}W(t))\equiv0$, i.e., $w(t)$ is a
solution of
$\tfrac{{\rm d}}{{\rm d}t}{w}(t)+N(y(t),w(t))+[y(t),w(t)]_\mathfrak{g}=0$.
\end{proof}
\subsection{Proof of Theorem \ref{thm-1}}

The Riemann curvature formula for a left invariant spray structure $\mathbf{G}$ is given by Theorem C in \cite{Xu2021}. In particular, when restricted to $T_eG$, the {\it Riemann curvature operator} $R_\cdot(\cdot):(\mathfrak{g}\backslash\{0\})\times\mathfrak{g}
\rightarrow\mathfrak{g}$ satisfies
(see Corollary 4.1 in \cite{Xu2021}),
\begin{equation}\label{006}
R_y(w)=DN(\eta,y,w)-N(y,N(y,w))+N(y,[y,w]_\mathfrak{g})-
[y,N(y,w)]_\mathfrak{g},
\end{equation}
 where $DN(\eta,y,w)=\tfrac{{\rm d}}{{\rm d}t}|_{t=0}N(y+t\eta,w)$, i.e., it is the derivative of $N(\cdot,w)$ at $y$ in the direction of $\eta(y)$. See Proposition 3.2 in \cite{Hu2015-2} for (\ref{006}) in homogeneous Finsler geometry, and Lemma 5.1 in \cite{HM2021} for another homogeneous Riemannian curvature formula when $\eta(y)=2P(y)y$ for some smooth positive 1-homogeneous function $P(\cdot)$ on $\mathfrak{g}\backslash\{0\}$.

Let $c(t)$ be a geodesic on $(G,\mathbf{G})$ and denote $y(t)=(L_{c(t)^{-1}})_*(\dot{c}(t))$.
The following theorem indicates that $y(t)$
is an integral curve of $-\eta$ (see Theorem D in \cite{Xu2021}).
\begin{theorem} \label{thm-2}
Let $\mathbf{G}$ be a left invariant spray structure on the Lie group $G$ with the associated spray vector field $\eta$. Then for any open interval $(a,b)\subset\mathbb{R}$ containing $0$, there is a one-to-one correspondence between the following two sets:
\begin{enumerate}
\item the set of all curves $c(t)$ on $G$, with $t\in(a,b)$ and $c(0)=e$, which are geodesics for $\mathbf{G}$;
\item the set of all $y(t)$ on $\mathfrak{g}\backslash\{0\}$, with $t\in(a,b)$, which are integral curves of $-\eta$.
\end{enumerate}
The correspondence between these two sets is given by
$y(t)=(L_{c(t)^{-1}})_*(\dot{c}(t))$.
\end{theorem}

Let $W(t)$ be a linearly parallel vector field along $c(t)$ and denote $w(t)=(L_{c(t)^{-1}})_*(\dot{c}(t))$, which is viewed as a smooth vector field along the curve $y(t)$ on $\mathfrak{g}\backslash\{0\}$. By Theorem \ref{main-thm-1},
we have $
\tfrac{{\rm d}}{{\rm d}t}w(t)+N(y(t),w(t))+[y(t),w(t)]_\mathfrak{g}=0$.
Now we prove Theorem \ref{thm-1}, i.e., an alternative interpretation of (\ref{006}).

\bigskip\begin{proof}[Proof of Theorem \ref{thm-1}]
By (\ref{007}) in Remark \ref{remark}, we have
\begin{eqnarray*}
[-\eta,w(t)]=\tfrac{{\rm d}}{{\rm d}t}w(t)+D\eta(y(t),w(t))
=\tfrac{{\rm d}}{{\rm d}t}w(t)+2N(y(t),w(t))+[y(t),w(t)]_\mathfrak{g},
\end{eqnarray*}
in which $D\eta(y(t),w(t))=\tfrac{{\rm d}}{{\rm d}s}|_{s=0}\eta(y(t)+sw(t))=2N(y(t),w(t))+[y(t),w(t)]_\mathfrak{g}$.
By Theorem \ref{main-thm-1}, $
\tfrac{{\rm d}}{{\rm d}t}w(t)+N(y(t),w(t))+[y(t),w(t)]=0$,
so
$N(t)=N(y(t),w(t))=[-\eta,w(t)]=-[\eta,w(t)]$.

Using Theorem \ref{main-thm-1}, Theorem \ref{thm-2} and the linearity of $N(y,w)$ for the $w$-entry,
we have
\begin{eqnarray}\label{008}
\tfrac{{\rm d}}{{\rm d}t}N(y(t),w(t))&=&
DN(-\eta,y(t),w(t))+N(y(t),\tfrac{{\rm d}}{{\rm d}t}w(t))
\nonumber\\
&=&-DN(\eta,y(t),w(t))+N(y(t),-N(y(t),w(t))-[y(t),w(t)]_\mathfrak{g})
\nonumber\\
&=&-DN(\eta,y(t),w(t))-N(y(t),N(y(t),w(t)))\nonumber\\& &-
N(y(t),[y(t),w(t)]_\mathfrak{g}).
\end{eqnarray}
Here $DN(\eta,y(t),w(t))$ is the derivative of $N(\cdot,w(t))$ at $y(t)$
in the direction of $w(t)$, for each fixed value of $t$.
So
(\ref{007}) and (\ref{008}) imply
\begin{eqnarray*}
[\eta,N(t)]&=&-[-\eta,N(t)]=-\tfrac{{\rm d}}{{\rm d}t}N(y(t),w(t))-
D\eta(y(t),N(y(t),w(t)))\\
&=&DN(\eta,y(t),w(t))+N(y(t),N(y(t),w(t)))+N(y(t),[y(t),w(t)]_\mathfrak{g})\\
& &-(2N(y(t),N(y(t),w(t)))+[y(t),N(y(t),w(t))]_\mathfrak{g})\\
&=&DN(\eta,y(t),w(t))-N(y(t),N(y(t),w(t)))+N(y(t),[y(t),w(t)]_\mathfrak{g})
\\& &-[y(t),N(y(t),w(t))]_\mathfrak{g}\\
&=&R_{y(t)}(w(t))=R(t).
\end{eqnarray*}
This ends the proof of Theorem \ref{thm-1}.
\end{proof}
\subsection{Landsberg curvature and S-curvature for a left invariant Finsler metric}
\label{subsection-2-3}
Theorem \ref{main-thm-1} and Theorem \ref{thm-2} can be used to prove other curvature formulas for $(G,\mathbf{G})$. Here we take
the Landsberg and S-curvature for a left invariant Finsler metric for example.

The Landsberg curvature $L$ for a Finsler metric $F$ can be calculated
by
\begin{equation}\label{016}
L_{\dot{c}(t)}(W(t),W(t),W(t))=
\tfrac{{\rm d}}{{\rm d}t}C_{\dot{c}(t)}(W(t),W(t),W(t)),
\end{equation}
in which $c(t)$ is a geodesic, $W(t)$ is linearly parallel along $c(t)$, and $C_\cdot(\cdot,\cdot,\cdot)$ is the Cartan tensor (see (7.16) in \cite{Sh2001-2}). When $F$ is left invariant,
Theorem \ref{main-thm-1} and Theorem \ref{thm-2} translate (\ref{016}) to
\begin{eqnarray*}
L_{y(t)}(w(t),w(t),w(t))&=&
\tfrac{{\rm d}}{{\rm d}t}C_{y(t)}(w(t),w(t),w(t))\\
&=&(\tfrac{{\rm d}}{{\rm d}t}C_{y(t)})(w(t),w(t),w(t),\eta(y(t)))
+3C_{y(t)}(\tfrac{{\rm d}}{{\rm d}t}w(t),w(t),w(t))\\
&=&-C_{y(t)}(w(t),w(t),w(t),\eta(y(t)))\\& &-
3C_{y(t)}(N(y(t),w(t))+[y(t),w(t)],w(t),w(t)),
\end{eqnarray*}
where $y(t)=(L_{c(t)^{-1}})_*(\dot{c}(t))$ is an integral curve of $-\eta$ and $w(t)=(L_{c(t)^{-1}})_*(W(t))$ is a solution of (\ref{011}). It verifies the Landsberg curvature formula in Proposition
6.1 in \cite{Hu2015-1} for a left invariant Finsler metric, i.e.,
$$L_y(w,w,w)=3C_y(w,w,[w,y]-N(y,w))-C_y(w,w,w,\eta(y)).$$

The S-curvature for a Finsler metric $F$ and a smooth measure ${\rm d}\mu$ can be calculated by
\begin{equation}\label{017}
S(c(t),\dot{c}(t))=\tfrac{{\rm d}}{{\rm d}t}\ln\det(\langle W_i(t),W_j(t)\rangle_{\dot{c}(t)})-\tfrac{{\rm d}}{{\rm d}t}
\ln{|\omega(W_1(t),\cdots,W_n(t))|},
\end{equation}
in which $c(t)$ is a geodesic, $\{W_1(t),\cdots,W_n(t)\}$ is any frame along $c(t)$, and $\langle\cdot,\cdot\rangle_{\cdot}$ is the fundamental tensor of $F$ \cite{XMYZ2020}. When $F$ is left invariant, ${\rm d}\mu=\omega$ is a left invariant volume form, and $W_i(t)$'s are taken to be linearly parallel along $c(t)$ and orthonormal with respect to $\langle\cdot,\cdot\rangle_{\dot{c}(t)}$,
the first summand in the right of (\ref{017}) vanishes, and
Theorem \ref{main-thm-1} implies
\begin{eqnarray*}
S(e,y(t))&=&S(c(t),\dot{c}(t))=
-\tfrac{{\rm d}}{{\rm d}t}\ln|\omega(W_1(t),\cdots,W_n(t))|\\& =&
-\tfrac{{\rm d}}{{\rm d}t}\ln|\omega(w_1(t),\cdots,w_n(t))|\\
&=&-\tfrac{\omega(\tfrac{{\rm d}}{{\rm d}t}w_1(t),w_2(t),\cdots,w_n(t))+\cdots+
\omega(w_1(t),\cdots,w_{n-1}(t),\tfrac{{\rm d}}{{\rm d}t}w_n(t))
}{\omega(w_1(t),\cdots,w_n(t))}\\
&=&-\sum_{i=1}^n\langle\tfrac{{\rm d}}{{\rm d}t}w_i(t),w_i(t)\rangle_{y(t)}\\
&=&\sum_{i=1}^n\langle N(y(t),w_i(t))+[y(t),w_i(t)],w_i(t)
\rangle_{y(t)}\\&=&\mathrm{Tr}_\mathbb{R}N(y(t),\cdot)+
\mathrm{Tr}_\mathbb{R}\mathrm{ad}(y(t)).
\end{eqnarray*}
in which $y(t)=(L_{c(t)^{-1}})_*(\dot{c}(t))$ and $w_i(t)=(L_{c(t)^{-1}})_*(W_i(t))$ for each $i$. It verifies the S-curvature formula in Proposition 6.1 of \cite{Hu2015-1}, i.e.,
$$S(y)=\mathrm{Tr}_\mathbb{R}N(y,\cdot)+
\mathrm{Tr}_\mathbb{R}\mathrm{ad}(y).$$

\section{Nonlinear parallel translation along a smooth curve}
\subsection{Proof of Theorem \ref{main-thm-2}}
Let $G$ be a connected Lie group endowed with a left invariant
spray structure $\mathbf{G}=\mathbf{G}_0-\mathbf{H}=u^i\widetilde{U}_i-
\mathbf{H}^i\partial_{u^i}$  with the spray vector field
$\eta$, and $c(t)$ with $t\in(a,b)$ a smooth curve on $G$ which is simple (i.e., it has no self intersection) and has nonvanishing $\dot{c}(t)$
for all values of $t$.

We denote $\dot{c}(t)=w^i(t)U_i(c(t))$, then its horizonal lifting
$\widetilde{\dot{c}(t)}^\mathcal{H}
=w^i(t)\widetilde{U}^\mathcal{H}_i|_{T_{c(t)}G\backslash\{0\}}$
is a smooth tangent vector field on the imbedded submanifold
$S=\cup_{t\in(a,b)}(T_{c(t)}G\backslash\{0\})$ of $TG\backslash0$.
Lemma 3.2 in \cite{Xu2021} provides a formula for $\widetilde{U}^\mathcal{H}_i$, i.e.,
\begin{lemma} \label{lemma-3}
The horizonal lifting of $U_q$ is $\widetilde{U}^\mathcal{H}_q=\widetilde{U}_q
-(\tfrac12\tfrac{\partial}{\partial u^q}\mathbf{H}^i-\tfrac12 u^j c^i_{qj})\partial_{u^i}.$
\end{lemma}
Its proof uses (\ref{002}), and a very similar calculation as (\ref{005}).

Using Lemma \ref{lemma-3} and (2) of lemma \ref{lemma-2}, we get
the decomposition
\begin{eqnarray}
\widetilde{\dot{c}(t)}^\mathcal{H}
&=&w^i(t)\widetilde{U}^\mathcal{H}_i|_{T_{c(t)}G\backslash\{0\}}
=w^i(t)(\widetilde{U}_i-
(\tfrac12\tfrac{\partial}{\partial u^i}\mathcal{H}^j-
\tfrac12u^kc^j_{ik})\partial_{u^j})|_{T_{c(t)}G\backslash\{0\}}
\nonumber\\
&=&w^i(t)\phi^j_i(c(t))\widetilde{V}_j
|_{T_{c(t)}G\backslash\{0\}}+
(\tfrac12w^i(t)u^pc^j_{pi}-\tfrac12
w^i(t)\tfrac{\partial}{\partial u^i}\mathcal{H}^j
)\partial_{u^j}|_{T_{c(t)}G\backslash\{0\}}\label{010}
\end{eqnarray}
at each $(c(t),y)\in S$ with $y=u^iU_i(c(t))\in T_{c(t)}G\backslash\{0\}$.
Both summands in the right side of (\ref{010}) are smooth vector fields on $N$. In particular,
the first one, $w^i(t)\phi^j_i(c(t))\widetilde{V}_j|_{T_{c(t)}G\backslash\{0\}}$ lifts $\dot{c}(t)$.

On $S$, we have the global coordinate $(t,u^1,\cdots,u^n)$ for $y=u^iU_i(c(t))\in T_{c(t)}M\backslash\{0\}$, and the corresponding global frame
$\{\partial_t,\partial_{u^1},\cdots,\partial_{u^n}\}$.  Using this frame, $\widetilde{\dot{c}(t)}^{\mathcal{H}}$ can be presented as
\begin{lemma}\label{lemma-4}
Using the global frame $\{\partial_t,\partial_{u^1},\cdots,\partial_{u^n}\}$ on $S$, we have
$$\widetilde{\dot{c}(t)}^\mathcal{H}=\partial_t+(-\tfrac12
w^i\tfrac{\partial}{\partial u^i}\mathcal{H}^j+\tfrac12w^i(t)u^pc^j_{pi}
)\partial_{u^j},$$
for each $(c(t),y)\in S$ with $y=u^iU_i(c(t))\in T_{c(t)}G\backslash\{0\}$.
\end{lemma}
 \begin{proof} We only need to prove $\partial_t=w^i(t)\phi^j_i(c(t))\widetilde{V}_j
|_{T_{c(t)}G\backslash\{0\}}$ on $N$. Notice that
the smooth vector field on $N$ which lifts
$\dot{c}(t)$ and keeps all $u^i$'s invariant is unique.
Obviously $\partial_t$ on $N$ is such a lifting.
The left invariancy of the $u^i$'s implies
$\widetilde{V}_ju^i=0,\forall i,j$. Together with (\ref{010}), it implies $w^i(t)\phi^j_i(c(t))\widetilde{V}_j|_{T_{c(t)}G\backslash\{0\}}$ is also such a lifting. These two liftings must be the same.
\end{proof}\bigskip

 \begin{proof}[Proof of Theorem \ref{main-thm-2}]
Since Theorem \ref{main-thm-2} is a local result, we only need to prove it in the case that $c(t)$  is a simple smooth curve. We may further assume $c(t)$ is defined for $t\in(a,b)$ and it has nonvanishing $\dot{c}(t)$ everywhere, because continuity can help us with the rest.

Using the global coordinate $(t,u^1,\cdots,u^n)$ on $S$,
a curve $(c(t),Y(t))$ in $S$ with $Y(t)=u^i(t)U_i(c(t))\in T_{c(t)}G\backslash\{0\}$ can be represented as $(t,u^1(t),\cdots,u^n(t))$. Then Lemma \ref{lemma-4} indicates that
$Y(t)$ is nonlinearly parallel, i.e., $(c(t),Y(t))$ is an integral curve of $\widetilde{\dot{c}(t)}^\mathcal{H}$, iff
\begin{equation}\label{014}
\tfrac{{\rm d}}{{\rm d}t}u^j(t)=\tfrac12w^i(t)u^pc^j_{pi}-\tfrac12
w^i\tfrac{\partial}{\partial u^i}\mathcal{H}^j,\quad\forall j.
\end{equation}
Using left translations, i.e., $y(t)=(L_{c(t)^{-1}})_*$ $(Y(t))=u^i(t)e_i$ and $w(t)=(L_{c(t)^{-1}})_*(\dot{c}(t))=w^i(t)e_i$, (\ref{014}) is translated to
$\tfrac{{\rm d}}{{\rm d}t}y(t)+N(y(t),w(t))=0$.
This ends the proof.
\end{proof}

\subsection{Two questions related to Landsberg Conjecture and holonomy}

\label{subsection-4-2}

Theorem \ref{main-thm-2} and Theorem \ref{cor} imply for any left invariant spray structure $\mathbf{G}$ on the Lie group $G$,
the following space of smooth vector fields on $\mathfrak{g}\backslash\{0\}$,
$\mathcal{N}=\{N(\cdot,w),\forall w\in\mathfrak{g}\}$,
contains all information of nonlinear parallel translations.
Naturally we would like to know more about the Lie algebra $\mathfrak{H}$ that $\mathcal{N}$ generates using the canonical bracket between smooth vector fields.

Here are some examples.
When $\mathbf{G}=\mathbf{G}_0=u^i\widetilde{U}_i$, we have $\eta=0$ and $N(y,w)=-\tfrac12[y,w]$. It is easy to check $[-2N(\cdot,w_1),-2N(\cdot,w_2)]=-2N(\cdot,[w_1,w_2])$, so in this case $\mathfrak{H}=\mathcal{N}$
is isomorphic to $\mathfrak{g}/\mathfrak{c}(\mathfrak{g})$. More generally, when the left invariant spray structure $\mathbf{G}$ is affine (see Definition 6.1.1 in \cite{Sh2001-1}), the associated spray vector field $\eta$ is quadratic. Then $\mathfrak{H}$ is a finite dimensional subalgebra in $\mathfrak{gl}(\mathfrak{g},\mathbb{R})$.

Above examples suggest we ask
\begin{question}\label{que-1}
Is there an example of left invariant spray structure
$\mathbf{G}$ such that $\mathbf{G}$ is not affine and $\dim\mathfrak{H}$ is finite?
\end{question}

Finsler geometry provides another motivation for Question
\ref{que-1}.
In Finsler geometry, a metric is called a {\it Berwald metric} if its induced spray structure is affine \cite{Sh2001-2},
and it is called a {\it Landsberg metric} if all
nonlinearly parallel translations are isometries for the Hessian metrics on the punctured tangent spaces \cite{Ic1978}. {\it Landsberg Problem} asks if there
exists a (regular) Landsberg metric which is not Berwald \cite{Ma1996}. See \cite{TN2021}\cite{XM2021} and the references therein for some recent progress on this problem. Theorem \ref{cor} implies that, if $\mathbf{G}$ is induced by a left invariant Landsberg metric, then $\mathfrak{H}$
is a Lie subalgebra in the space of all Killing vector fields for the Hessian metric of $F(e,\cdot)$ on $\mathfrak{g}\backslash\{0\}$, which must have a finite dimension. So Question \ref{que-1} may be viewed as a generalization for Landsberg Problem in the left invariant spray geometry.

Another natural question for $\mathfrak{H}$ is the following.

\begin{question}\label{que-2}
What is the relation between the Lie algebra $\mathfrak{H}$ and the
restricted holonomy group of $(G,\mathbf{G})$.
\end{question}

The restricted holonomy group $\mathrm{Hol}_0(G,\mathbf{G})$ of $(G,\mathbf{G})$ is the subgroup
of $\mathrm{Diff}(\mathfrak{g}\backslash0)$
generated by $\mathbf{P}^{\mathrm{nl}}_{c(0),c(1);c}$ for all
piecewise smooth curves $c(t):[0,1]\rightarrow G$ which are homotopic to a constant map and satisfy $c(0)=c(1)=e$.

When the left invariant spray structure $\mathbf{G}$ is induced by
a Riemannian metric $F$ on the Lie group $G$, $\mathrm{Hol}_0(G,\mathbf{G})$ is a compact Lie group.
Comparing Lemma 2.2 in \cite{Ko1957} and Definition 4 in \cite{Hu2015-2}, then we see that $N(\cdot,v)$ coincides with the linear operator $a_v$ in \cite{Ko1955} for each $v\in\mathfrak{g}$. So in this case, Theorem 4.5 in \cite{Ko1955} indicates $\mathfrak{H}=\mathrm{Lie}(\mathrm{Hol}_0(G,\mathbf{G}))$ when either $G$ is compact or $(G,F)$ is irreducible with nonvanishing Ricci curvature.

However, when the left invariant $\mathfrak{H}$ is more generic, very likely both $\mathfrak{H}$ and $\mathrm{Hol}_0(G,\mathbf{G})$ have infinite dimensions \cite{HMM2020}, making Question \ref{que-2} is much harder in this situation. \bigskip

\noindent
{\bf Acknowledgement}.
This paper is supported by Beijing Natural Science Foundation
(Z180004), National Natural Science
Foundation of China (12131012, 11821101). The author sincerely thanks Ming Li for helpful discussions.

\end{document}